\documentclass[a4paper,12pt]{article}
\usepackage[ansinew]{inputenc}
\usepackage{amsfonts}
\usepackage{latexsym}
\usepackage{amsmath}
\usepackage{amssymb}
\usepackage{multirow}
\usepackage{eepic}
\usepackage{graphicx}
\usepackage{color}
\usepackage{tikz}
\usepackage{geometry}

\newcommand{\U}{\underline}

\newcommand{\R}{\mathbb{R}}

\newcommand{\C}{\mathbb{C}}
\newcommand{\N}{\mathbb{N}}

\newtheorem{proposition}{Proposition}[section]

\newtheorem{defin}{Definition}[section]

\newtheorem{theorem}[defin]{Theorem}
\newtheorem{exa}{Example}[section]
\newenvironment{example}{\begin{exa}\rm}{\end{exa}}
\newtheorem{exas}{Examples}

\newtheorem{lemma}[defin]{Lemma}

\newtheorem{corollary}[defin]{Corollary}
\newenvironment{proof}
{\noindent{\it Proof.}}{\hfill $\Box$\par\vspace{2.5mm}}
\newtheorem{remark}{Remark}

\numberwithin{equation}{section}

\title{Difference radical in terms of shifting zero and applications to the Stothers-Mason theorem}
\vspace{0.5cm}

\author{Katsuya Ishizaki\footnote{\ Supported by JSPS KAKENHI Grant Number 20K03658}~
 and Zhi-Tao Wen\footnote{\ Supported by the National Natural Science Foundation of China (No.~11971288 and No.~11771090) and Shantou University SRFT (NTF18029)}
 }
\date{}

\begin{document}
\maketitle

\begin{abstract}
In this paper, we study the shifting zeros with its heights and an analogue to difference radical.
We focus on the Stothers-Mason theorem by using falling factorials. As applications, we discuss the difference version of the Fermat type functional equations. Some examples are given.


\medskip
\noindent
\textbf{2020MSC}: 39B10; 30D35.

\end{abstract}

\section{Introduction}
Let $P$ be a polynomial. The radical ${\rm rad}(P)$ is the product of distinct linear factors of $P$.
We study a difference analogue of radical of a complex polynomial, and consider a difference analogue of the Stothers-Mason theorem. As applications, we study the difference version of the Fermat type functional equations, see e.g.,~\cite{IKLT},~\cite{li2015}.

Let $a$, $b$ and $c$ be relatively prime polynomials such that not all of them are identically zero. The Stothers--Mason theorem~\cite{mason1984},~\cite{stothers1981}, see also e.g.,~\cite{Snyder} states that if they satisfy $a+b=c$, then
\begin{equation}\label{smt}
\max\{\deg(a),\deg(b),\deg(c)\}\leq \deg({\rm rad}(abc))-1.
    \end{equation}
A difference analogue of the Stothers--Mason theorem is seen in e.g.,~\cite{IKLT}.
One of the purpose of this paper is that we give an alternative approach to difference analogues of radical of polynomials in terms of shifting zero different from \cite{IKLT}.

\vspace{0.25cm}

We recall basic notations in difference calculus. Given a function $f$, we denote by $\Delta f(z)=f(z+1)-f(z)$ the difference operator. Let $n$ be a nonnegative integer. Define $\Delta^n f(z)=\Delta(\Delta^{n-1} f(z))$ for $n\geq 1$, and write $\Delta^0 f=f$.
Define $z^{\underline{0}}=1$ and
\begin{equation}
z^{\underline{n}}=z(z-1)\cdots(z-n+1)=n!\binom{z}{n},\quad n=1, 2, 3, \dots,\label{1.3}
\end{equation}
which is called a {\it falling factorial}, see e.g.,~\cite[Pages 6--7]{Boole},~\cite[Subsection~A.2]{Ishizaki-Wen},~\cite[Page 25]{M-Thomson1933}. This yields $\Delta z^{\underline{n}}=(z+1)^{\underline{n}}-z^{\underline{n}}=nz^{\underline{n-1}}$, which corresponds to $(z^n)'=nz^{n-1}$ in the differential calculus.

Let $f$ be a complex function and let $n\in\N$. Consider products of the forms
$f^{\underline{n}}(z)=f(z)f(z-1)\cdots f(z-n+1)$, $n=1,2,3,\dots$, $f^{\underline{0}}=1$ and
$f^{\overline{n}}(z)=f(z)f(z+1)\cdots f(z+n-1)$, $n=1,2,3,\dots$, $f^{\overline{0}}=1$.
We call $f^{\underline{n}}$ a falling (descending) factorial expression, and call $f^{\overline{n}}$ a raising (ascending) factorial expression, see e.g.,~\cite[Page 25]{M-Thomson1933}.
The falling factorial $z^{\underline{n}}=z(z-1)\cdots (z-n+1)$ given in \eqref{1.3} is the falling factorial expression of $f(z)=z$.
\vspace{0.25cm}

In Section~\ref{shifting zeros}, we give the definition of shifting zero and study some properties of it with difference calculus. We discuss the difference radical by using shifting zero and focus on the difference analogue of the Stothers-Mason theorem by using falling factorials in Section~\ref{dirad}. Section~\ref{edsm} is concerned with an extension of difference analogue of the Stothers--Mason theorem. We treat in Section~\ref{psf} polynomial solutions of difference version of the Fermat type functional equations.

\section{Shifting zero}\label{shifting zeros}

Let $f$ be analytic in a domain $G\subset\C$ and let $n\in\N$.
Suppose that $z_0, z_0+1,\ldots, z_0+n\in G$.
The point $z_0$ is called a \emph{shifting zero of $f(z)$ with height $n$ in $G$} provided $f(z_0)$ and all difference $\Delta^k f(z_0)$ vanish for every $0\leq k<n$, but $\Delta^n f(z_0)\neq 0$.
In particular, if $z_0$ is a shifting zero of $f$ with height 1, then we also call $z_0$ is a shifting zero of $f$ with simple height.

\begin{example}
The function $f(z)=2^z=e^{z\ln 2}$ has no zeros, which has no shifting zeros. Let $g$ be defined as $g(z)=z^2(z-1)(z-2)2^z$, which can be written
    $$
   g(z)=z^{\underline{3}}\cdot z2^z=(z-1)^{\underline{2}}\cdot z^22^z=(z-2)^{\underline{1}}\cdot z^2(z-1)2^z.
    $$
Then $z=0$ is a shifting zero of $g$ with height $3$ in $\C$. In addition, $z=1$ and $z=2$ are the shifting zeros of $g$ with height 2 and 1 in $\C$, respectively.

Consider the Euler Gamma function $\Gamma$. It is known that $1/\Gamma$ is a transcendental entire function that has simple zeros at $0, -1, -2, \dots$. This gives that for any $n\in\N$, $z=-n$ is a shifting zero of $1/\Gamma$ with height $n+1$ in $\C$, see e.g.,~\cite[Page 236]{WW1927}.
\end{example}

\begin{example}
Consider a polynomial $f(z)=z^2(z-1)^3$. It gives us that $z=0$ is a zero of $f$ with order 2, and $z=1$ is a zero of $f$ with order 3. We could also see that $z=0$ is a shifting zero of $f$ with height 2 in $\C$, and $z=1$ is a shifting zero of height $1$ in $\C$. In fact, $f$ is of the form
    $$
    f(z)=(z^{\underline{2}})^2(z-1)=z^{\underline{2}}\cdot z(z-1)^2.
    $$
\end{example}

\begin{remark}
A function $f(z)$ may have a shifting zero with infinite height in $\C$. For example, any integer point is the zero of $f(z)=e^{2\pi iz}-1$. If a function $f$ has a shifting zero with infinite height, then the order of growth of $f$ is at least 1. It is important to study the order of growth of solutions of nonlinear difference equations. For example, let $w$ be a solution of a difference Riccati equation
    \begin{equation}\label{AR.eq}
    w(z+1/2)=\frac{A(z)}{w(z)},
    \end{equation}
where $A(z)$ is a nonconstant rational function. Suppose that all zeros and poles of $A(z)$ are lying on $|z|<R$. Let $z=z_0$ be a zero of $w(z)$ such that $|z_0|>R$, it gives $z_0+1, z_0+2,\ldots$ are zeros of $w(z)$.
Then $z_0$ is a shifting zero of $w$ with infinite height in $\C$, otherwise, $w$ has at most finitely many zeros and poles. It shows that there does not exist any transcendental meromorphic solutions of \eqref{AR.eq} of order less than 1.
\end{remark}

It is well known that $z=z_0$ is a zero of an analytic function $f$ with order $n$ if and only if $f$ is of
the form $f=(z-z_0)^ng(z)$, where $g(z)$ is analytic at $z=z_0$ and $g(z_0)\neq 0$. The difference analogue of this result is given as follows.

\begin{theorem}\label{zero.theorem}
Let $f$ be analytic in the domain $G\subset\C$ and let $n\in\N$. Suppose that $z_0, z_0+1,\ldots, z_0+n\in G$.
The point $z_0$ is a shifting zero of $f(z)$ with height $n$ in $G$ if and only if there exists an analytic function $g(z)$ in $G$ such that
    \begin{equation}\label{fzero.eq}
    f(z)=(z-z_0)^{\underline{n}}g(z)
    \end{equation}
and $g(z_0+n)\neq 0$.
\end{theorem}

\begin{proof}
Let $z_0$ be a shifting zero of $f$ with height $n$ in $G$.
It is known that for any $z\in G$ and $k\in \N$
\begin{equation}
f(z+k)=\sum_{j=0}^k \begin{pmatrix} k  \\ j \end{pmatrix}\Delta^jf(z).
\label{4.0012}
\end{equation}
Since $\Delta^j f(z_0)$ vanish for $0\leq j\leq k<n$, \eqref{4.0012} implies that $f(z_0+k)=0$ for $0\leq k<n$. Further, setting $z=z_0$ and $k=n$ in \eqref{4.0012}, we see that $f(z_0+n)=
\Delta^nf(z_0)\ne0$. Hence we can write $f(z)$ in the form \eqref{fzero.eq} and $g(z_0+n)\neq 0$.
%
%
%
%
%

Let us prove another direction as follows. We also have for any $z\in \C$ and $k\in \N$
\begin{equation}
\Delta^k f(z)=\sum_{j=0}^k \begin{pmatrix} k\\ j \end{pmatrix}(-1)^{k-j}f(z+j).\label{4.002}
\end{equation}
Assume that $f$ is of the form \eqref{fzero.eq} and $g(z_0+n)\neq 0$. Then by means of \eqref{4.002}, $f(z_0)$ and all difference $\Delta^k f(z_0)$ vanish for $0<k<n$, but $\Delta^n f(z_0)\neq 0$. That proves $z_0$ is a shifting zero of $f$ with height $n$ in $G$. For formulas \eqref{4.0012} and \eqref{4.002}, see e.g.,~\cite[Page 14]{KP2001},~\cite[Page 4]{Norlund1924}.
\end{proof}



\begin{corollary}\label{n-1.cor}
Let $f$ be analytic in the domain $G\subset\C$ and let $n\in\N$. Suppose that $z_0, z_0+1,\ldots, z_0+n\in G$.
If $z_0$ is a shifting zero of $f$ with height $n$ in $G$, then $z_0$ is a shifting zero of $\Delta f$ with height $n-1$.
\end{corollary}

\begin{proof}
Since $z_0$ is a shifting zero of $f$ with height $n$ in $G$, there
exists an analytic function $g(z)$ in $G$ such that
    $$
    f(z)=(z-z_0)^{\underline{n}}g(z)
    $$
and $g(z_0+n)\neq 0$. Moreover, we have
    \begin{equation*}
    \begin{split}
    \Delta f(z)&=(z+1-z_0)^{\underline{n}}g(z+1)
    -(z-z_0)^{\underline{n}}g(z)\\
    &=(z-z_0)^{\underline{n-1}}\big((z+1-z_0)g(z+1)-(z-z_0-n+1)g(z)\big)\\
    &=(z-z_0)^{\underline{n-1}}Q(z).
    \end{split}
    \end{equation*}
We see that $Q(z_0+n-1)=ng(z_0+n)\neq 0$. It gives us that $z_0$ is a shifting zero of $\Delta f$ with height $n-1$ from Theorem~\ref{zero.theorem}.
\end{proof}

The next theorem puts the focus on an entire function of the order of growth less than 1, which is a particular case of Theorem~\ref{zero.theorem}.
Let $f$ be an entire function of the order of growth $\rho(f)<1$. By means of~\cite[Theorem 1.1]{Ishizaki-Wen}, it follows
that $f$ can be written by a binomial (factorial) series in $\C$
\begin{equation}
f(z)=\sum_{j=0}^\infty a_j(z-z_0)^{\U{j}}=a_0+a_1(z-z_0)^{\U{1}}+a_2(z-z_0)^{\U{2}}+\cdots+a_j(z-z_0)^{\U{j}}+\cdots,\label{4.003}
\end{equation}
where the sequence $\{a_j\}$ satisfies
\begin{equation}
\chi(\{a_j\})=\limsup_{j\to\infty}\frac{j\log j}{-\log |a_j|}<1.\label{4.004}
\end{equation}

\begin{theorem}
Let $f$ be an entire function of the order of growth less than 1, and let $n\in\N$. The point $z_0$ is a shifting zero of $f(z)$ with height $n$ in $\C$ if and only if there exists an entire function $g$ of the order of growth less than 1 such that
    \begin{equation}\label{fzero2.eq}
    f(z)=(z-z_0)^{\underline{n}}g(z),
    \end{equation}
and $g(z_0+n)\neq 0$. Furthermore, $f$ and $g$ are represented by binomial series
$$
f(z)=\sum_{j=n}^\infty a_j(z-z_0)^{\U{j}}\quad\text{and}\quad g(z)=\sum_{j=0}^\infty a_{n+j}(z-z_0-n)^{\U{j}}\ .
$$
\end{theorem}

\begin{proof}
We write $f$ as in \eqref{4.003}.
Suppose that $z_0$ is a shifting zero with height $n$. We have $\Delta^k f(z_0)=0$ for every $0\leq k\leq n-1$, and $\Delta^n f(z_0)\neq 0$. Then $a_k=0$ for every $0\leq k\leq n-1$ and $a_n\neq 0$, which implies
    \begin{equation*}
    \begin{split}
    f(z)=\sum_{j=n}^\infty a_j(z-z_0)^{\U{j}}&=a_n(z-z_0)^{\U{n}}+a_{n+1}(z-z_0)^{\U{n+1}}+a_{n+2}(z-z_0)^{\U{n+2}}+\cdots\\
    &=(z-z_0)^{\U{n}}\left(a_n+a_{n+1}(z-z_0-n)^{\U{1}}+a_{n+2}(z-z_0-n)^{\U{2}}+\cdots\right).
    \end{split}
    \end{equation*}
We set $g$ as
\begin{equation}
g(z)=\sum_{j=0}^\infty a_{n+j}(z-z_0-n)^{\U{j}}=a_n+a_{n+1}(z-z_0-n)^{\U{1}}+a_{n+2}(z-z_0-n)^{\U{2}}+\cdots.\label{4.005}
\end{equation}
Then $g(z)$ is an entire function of the order of growth $\rho(g)<1$
 and $g(z_0+n)=a_n\neq 0$.
Therefore, we proved that $f$ is of the form \eqref{fzero2.eq}.
\vspace{0.25cm}

In what follows, we proceed to prove another direction of this theorem. Let $g$ be an entire function of the order of growth $\rho(g)<1$ such that $g(z_0+n)\neq 0$.
Then by \cite[Theorem 1.1]{Ishizaki-Wen}, $g$ can be supposed as in \eqref{4.005} satisfying \eqref{4.004}.
If $f$ is of the form \eqref{fzero2.eq}, then
    $$
    f(z)=(z-z_0)^{\U{n}}g(z)=(z-z_0)^{\U{n}}\left(\sum_{j=0}^\infty a_{n+j}(z-z_0-n)^{\U{j}}\right)
    =\sum_{j=n}^\infty a_j(z-z_0)^{\U{j}},
    $$
where $a_n\neq 0$. From \cite[Theorem 1.1]{Ishizaki-Wen}, it gives that $f$ is entire function of the order of growth $\rho(f)<1$.
Moreover, it is clear that $\Delta^k f(z_0)=0$ for every $0\leq k\leq n-1$ and $\Delta^n f(z_0)\neq 0$.
Therefore, $z_0$ is a shifting zero of $f$ with height $n$ in $\C$. We thus proved our assertion.
\end{proof}

\section{Difference radical}\label{dirad}

Let $P$ be a polynomial with degree $p$, and $z_1$ be a shifting zero of $P$ with height $n_1$ and $P(z_1-1)\neq 0$. Theorem~\ref{zero.theorem} states that there exists a polynomial $P_1$ with degree $p-n_1$ such that $P(z)=(z-z_1)^{\underline{n_1}}P_1(z)$.
Let $z_2$ be a shifting zero of $P_1$ with height $n_2$ and $P(z_2-1)\neq 0$.
Then there exists a polynomial $P_2$ with degree $p-n_1-n_2$ such that $P_1(z)=(z-z_2)^{\underline{n_2}}P_2(z)$.
Repeating this argument for finitely many times, we see that $P(z)$ can be written uniquely as
\begin{equation}
P(z)=A\prod_{j=1}^N(z-z_j)^{\underline{n_j}},\label{5.001}
\end{equation}
where $A$ is a nonzero constant, and $p=n_1+\cdots+n_N$.
Note that it is possible that $z_j=z_k$ even though $j\ne k$ in \eqref{5.001}.
We define the \emph{difference radical ${\rm rad}_{\Delta}(P)$} by product of these linear factors, i.e.,
\begin{equation}
{\rm rad}_{\Delta}(P)=\prod_{j=1}^N(z-z_j).\label{5.002}
\end{equation}
For example, if $g(z)=z^2(z-1)(z-2)=z^{\underline{1}}z^{\underline{3}}$, then
${\rm rad}_{\Delta}(g)=z^2$.


In \cite{IKLT} the definition of $\kappa$-difference radical ${\rm r}\tilde{\rm a}{\rm d}_{\kappa}(P)$ of a polynomial $P$ is given as
\begin{equation}
    {\rm r}\tilde{\rm a}{\rm d}_{\kappa}(P)=\prod_{w\in\C}(z-w)^{d_{\kappa}(w)},\label{5.0039}
\end{equation}
 where
\begin{equation}
d_{\kappa}(w)=d_{\kappa}(w,P)=\text{ord}_w(P)-\min\{\text{ord}_w(P),\text{ord}_{w+\kappa}(P)\},\label{5.0033}
\end{equation}
with $\text{ord}_w(P)\geq 0$ being the order of zero of the polynomial $P$ at $w\in\C$.
We note here even if $\kappa=1$, the definition of $\kappa$-difference radical is not the same as our definition above.
For example, if $g(z)=z^2(z-1)(z-2)$, then ${\rm r}\tilde{\rm a}{\rm d}_{1}(g)=z(z-2)$, which is different from ${\rm rad}_{\Delta}(g)=z^2$.
On the other hand, the ideas in~\cite{IKLT} work for considering some problems even though the definitions are different. Actually, there are several properties in common.
For the sake of simplicity, we write $d_1(w)$ given in \eqref{5.0033} as $d(w)$, and write
${\rm r}\tilde{\rm a}{\rm d}_1(P)$ given in \eqref{5.0039} as ${\rm r}\tilde{\rm a}{\rm d}(P)$.
We have the following proposition.

\begin{proposition}\label{radical degree}
Let $P$ be a polynomial. Then $\deg{\rm rad}_{\Delta}(P)=\deg{\rm r}\tilde{\rm a}{\rm d}(P)$.
\end{proposition}
\begin{proof}
When $P$ is a constant, the assertion holds. We assume that $P$ is a nonconstant polynomial.
Let $Z_P$ be the set of zeros of $P$, i.e., $Z_P=\{w\ | P(w)=0\}$ considering the multiplicities.
In other words, if $w$ is a zero of $P$ then $w$ appears ${\rm ord}_w(P)$ times in $Z_P$.

Firstly we consider $\deg{\rm rad}_{\Delta}(P)$.
The process to determine \eqref{5.001} is translated as follows. Choose $z_1\in Z_P$ satisfying $z_1-1\not\in Z_P$. If $z_1+1\in Z_P$ then we adopt $z_1+1$. Repeating this procedure, we adopt $z_1+k\in Z_P$, $k\in \N$ and stop $z_1+n_1-1\in Z_P$ when $z_1+n_1\not\in Z_P$. Then we find a finite sequence $L_1=\{z_1,z_1+1,\dots, z_1+n_1-1\}$. We may have $L_1=\{z_1\}$ in case $z_1+1\not\in Z_P$.
Next Choose $z_2\in Z_P\setminus L_1$ satisfying $z_2-1\not\in Z_P\setminus L_1$.
It may be possible $z_1=z_2$ when $z_1$ is a zero of $P$ with multiplicity $\geq2$.
If $z_2+1\in Z_P\setminus L_1$ then we adopt $z_2+1$. Repeating this procedure, we adopt $z_2+k\in Z_P\setminus L_1$, $k\in \N$ and stop $z_2+n_2-1\in Z_P\setminus L_1$ when $z_2+n_2\not\in Z_P\setminus L_1$. Then we find a finite sequence $L_2=\{z_2,z_2+1,\dots, z_2+n_2-1\}$. Continuing this process, we divide $Z_P$ into $N\in\N$ finite sequences, namely,
$$
Z_P=\bigcup_{j=1}^{N}L_j=\bigcup_{j=1}^{N}\{z_j,z_j+1,\dots, z_j+n_j-1\},
$$
with $z_j-1, z_j+n_j\not\in Z_P\setminus\left(\bigcup_{k=1}^{j-1}L_k\right)$ for any $j$. We see that $N$, say the number of $L_j$'s, coincides with ${\rm rad}_\Delta (P)$.

Secondly we observe $\deg{\rm r}\tilde{\rm a}{\rm d}(P)$. By \eqref{5.0033}, we have for $w\in Z_P$
\begin{equation}
d(w)=\text{ord}_w(P)-\min\{\text{ord}_w(P),\text{ord}_{w+1}(P)\}=\max\{0,\text{ord}_w(P)-\text{ord}_{w+1}(P)\}.\label{5.0037}
\end{equation}
If $w\in Z_P$ satisfies $\text{ord}_w(P)-\text{ord}_{w+1}(P)>0$, then $w$ contributes to ${\rm r}\tilde{\rm a}{\rm d}(P)$.
Choose $w_1\in Z_P$ satisfying $w_1+1\not\in Z_P$, which contributes to ${\rm r}\tilde{\rm a}{\rm d}(P)$.
If $w_1-1\in Z_P$ then we adopt $w_1-1$. Repeating this procedure, we adopt $w_1-k\in Z_P$, $k\in \N$ and stop $w_1-m_1+1\in Z_P$ when $z_1-m_1\not\in Z_P$. Then we find a finite sequence $K_1=\{w_1,w_1-1,\dots, w_1-m_1+1\}$. Note that $w_1$ can be chosen $\text{ord}_{w_1}(P)$ times in the first step in case $\text{ord}_{w_1}(P)\geq2$, and we may have $K_1=\{w_1\}$ in case $w_1-1\not\in Z_P$.
Next Choose $w_2\in Z_P\setminus K_1$ satisfying $w_2+1\not\in Z_P\setminus K_1$.
It may be possible $w_1=w_2$ when $w_1$ is a zero of $P$ with multiplicity $\geq2$ as mentioned above.
If $w_2-1\in Z_P\setminus K_1$ then we adopt $w_2-1$. Repeating this procedure, we adopt $w_2-k$, $k\in \N$ and stop $w_2-m_2+1\in Z_P\setminus K_1$ when $w_2-m_2\not\in Z_P\setminus K_1$. Then we find a finite sequence $K_2=\{w_2,w_2-1,\dots, w_2-m_2+1\}$. Continuing this process, we divide $Z_P$ into $M\in\N$ finite sequences, namely,
$$
Z_P=\bigcup_{j=1}^{M}K_j=\bigcup_{j=1}^{M}\{w_j,w_j-1,\dots, w_j-m_j+1\},
$$
with $w_j+1, w_j-m_j\not\in Z_P\setminus\left(\bigcup_{k=1}^{j-1}K_k\right)$ for any $j$.
Then we see that ${\rm r}\tilde{\rm a}{\rm d}(P)=\prod_{j=1}^{M}(z-w_j)=\prod_{w\in\C}(z-w)^{d_w}$.

The finite sequences $L_j\subset Z_P$ and $K_j\subset Z_P$ possess a common property that outsides of these sequences $z_j-1$, $z_j+n_j$, $w_j+1$, $w_j-m_j$ do not belong to $Z_P$. This division of $Z_P$ is uniquely determined, which implies $N=M$. This concludes that $\deg{\rm rad}_{\Delta}(P)=\deg{\rm r}\tilde{\rm a}{\rm d}(P)$.
\end{proof}

\begin{remark}\label{radrad}
Using the notations in the proof of Proposition~\ref{radical degree}, we have
$$
{\rm rad}_{\Delta}(P)=\prod_{j=1}^N(z-z_j)\quad \text{and}\quad {\rm r}\tilde{\rm a}{\rm d}(P)=\prod_{j=1}^M(z-w_j).
$$
Consider the example $g(z)=z^2(z-1)(z-2)$ which is treated just below \eqref{5.0033}. The zeros $\{0,0,1,2\}$ of $g$ are divided into two sequences $\{0\}$ and $\{0,1,2\}$. Thus we confirm that ${\rm rad}_{\Delta}(g)=z\cdot z=z^2$ and ${\rm r}\tilde{\rm a}{\rm d}(g)=z(z-2)$.
We observe another example $h(z)=(z+1)z^2(z-1)^3(z-2)^2(z-4)$.
The zeros $\{-1,0,0,1,1,1,2,2,4\}$ of $h$ are divided into four sequences $\{-1,0,1,2\}$, $\{0,1,2\}$, $\{1\}$ and $\{4\}$, which gives that ${\rm rad}_{\Delta}(g)=(z+1)z(z-1)(z-4)$ and ${\rm r}\tilde{\rm a}{\rm d}(g)=(z-1)(z-2)^2(z-4)$.
\end{remark}


Suppose that $f$ and $g$ are entire functions, and suppose that $z_1\in\C$ is a shifting zero of $f$ with height $m$ and $z_2\in\C$ is a shifting zero of $g$ with height $n$.
Then from Theorem~\ref{zero.theorem}, there exist entire functions $F$ and $G$ such that
$f(z)=(z-z_1)^{\underline{m_1}}F(z)$ and $g(z)=(z-z_2)^{\underline{n_1}}G(z)$, where $1\leq m_1\leq m$ and $1\leq n_1\leq n$. If
    $$
    f(z)g(z)=(z-z_0)^{\underline{m_1+n_1}}F(z)G(z),
    $$
where $z_0$ is $z_1$ or $z_2$, then $z-z_0$ is called the \emph{common shifting divisor} of $f$ and $g$, which is the analogue of classical common divisor.  If $z_0=z_1$, then $z_2=z_1+m_1$, and if $z_0=z_2$, then $z_1=z_2+n_1$.
For example, if
$f(z)=z(z-1)(z-2)$ and $g(z)=(z-2)(z-3)(z-4)$, then $z$, $z-1$, $z-2$ are the common shifting divisor of $f$ and $g$.
In addition, $z-2$ is the common divisor of $f$ and $g$. If $f$ and $g$ do not have any nonconstant shifting common divisors, then $f$ and $g$ are called \emph{shifting prime}.

\begin{remark}\label{shifting radical}
Let $P$ and $Q$ be nonconstant polynomials. In general, we have
\begin{equation}
\deg({\rm rad}_\Delta (PQ))\leq
\deg({\rm rad}_\Delta (P))+\deg({\rm rad}_\Delta (Q)).\label{5.0025}
 \end{equation}
By the definition of the shifting prime, the equality in \eqref{5.0025} holds if $P$ and $Q$ are shifting prime. On the other hand, even though the equality in \eqref{5.0025} holds, $P$ and $Q$ are not always shifting prime.
For example, consider $P(z)=z$ and $Q(z)=z(z-1)$. Then ${\rm rad}_\Delta (PQ)=z^2$, and ${\rm rad}_\Delta (P)=z$, ${\rm rad}_\Delta (Q)=z$, which gives the equality in \eqref{5.0025} holds. However, $P$ and $Q$ are not shifting prime. In fact, we can write $P(z)=z^{\underline{1}}\cdot 1$, $Q(z)=z^{\underline{1}}\cdot (z-1)$ and $P(z)Q(z)=z^{\underline{1+1}}\cdot z$.
\end{remark}

In the following, we denote the greatest common divisor of $f$ and $g$ by $\text{gcd}(f,g)$.
It is well known that if two entire functions $f$ and $g$ are prime, then $\text{gcd}(f,f')$ and $\text{gcd}(g,g')$ are also prime. We state the difference analogue of this result as follows.

\begin{lemma}\label{shifting prime}
Suppose that $f$ and $g$ are entire functions and shifting prime. Then ${\rm gcd}(f,\Delta f)$
and ${\rm gcd}(g,\Delta g)$ are also shifting prime. Moreover, ${\rm gcd}(f,\Delta f)$
and ${\rm gcd}(g,\Delta g)$ are prime.
\end{lemma}

\begin{proof}
Assume that $\text{gcd}(f,\Delta f)$ and $\text{gcd}(g,\Delta g)$ have a common shifting divisor. This leads that there exist $z_1$, $z_2\in\C$, $m$, $n\in\N$ and entire functions $J$ and $K$ such that
    \begin{equation}\label{Df.eq}
    \text{gcd}(f,\Delta f)=(z-z_1)^{\underline{m}}J(z),
    \end{equation}
and
    \begin{equation}\label{Dg.eq}
    \text{gcd}(g,\Delta g)=(z-z_2)^{\underline{n}}K(z).
    \end{equation}
We assume that $z_2=z_1+m$ without loss of generality. Then we have
    $$
    \text{gcd}(f,\Delta f)\text{gcd}(g,\Delta g)=
    (z-z_1)^{\underline{m+n}}J(z)K(z).
     $$
It follows from \eqref{Df.eq}, \eqref{Dg.eq} and Corollary \ref{n-1.cor} that
    $f(z)=(z-z_1)^{\underline{m+1}}J_1(z)$ and $g(z)=(z-z_2)^{\underline{n+1}}K_1(z)$, where $J_1$ and $K_1$ are entire functions. In this way, we have
    $$
    f(z)g(z)=(z-z_1)^{\underline{m+n+1}}(z-z_1+m)J_1(z)K_1(z).
    $$
It shows that $f$ and $g$ have a common shifting divisor $z-z_1$, which is a contradiction with our assumption. We thus proved the first assertion.

Assume that $\text{gcd}(f,\Delta f)$ and $\text{gcd}(g,\Delta g)$ are not prime.
Let $z-z_0$ be the common divisor of $\text{gcd}(f,\Delta f)$
and $\text{gcd}(g,\Delta g)$. Then $\text{gcd}(f,\Delta f)=(z-z_0)S(z)$ and $\text{gcd}(g,\Delta g)=(z-z_0)T(z)$, where $S$ and $T$ are entire functions.
By Corollary~\ref{n-1.cor}, we obtain that $f(z)=(z-z_0)^{\underline{2}}S_1(z)$ and $g(z)=(z-z_0)^{\underline{2}}T_1(z)$, where $S_1$ and $T_1$ are entire functions. It implies that $f$ and $g$ are not shifting prime, which is a contradiction to our assumption. Therefore, we proved Lemma~\ref{shifting prime}.
\end{proof}


Let $a$, $b$ and $c$ be relatively prime polynomials such that not all of them are identically zero.
If they satisfy $a+b=c$, then we have \eqref{smt} by the Stothers--Mason theorem.
For a polynomial $P$, we write $\tilde{n}=\deg({{\rm r}\tilde{\rm a}{\rm d}}(P))$, where ${\rm r}\tilde{\rm a}{\rm d}(P)$ is given in~\eqref{5.0039} with $\kappa=1$.
The difference analogue of the Stothers--Mason theorem is given in \cite[Theorem~3.1]{IKLT}, which states that
if $a+b=c$, then
\begin{equation}\label{mz}
\max\{\deg(a),\deg(b),\deg(c)\}\leq \tilde{n}(a)+\tilde{n}(b)+\tilde{n}(c)-1.
\end{equation}
The inequality \eqref{mz} is sharp, see~\cite[Example~3.2]{IKLT}, which is mentioned below in Remark~\ref{sharpm2}.
In general, $\tilde{n}(abc)\leq\tilde{n}(a)+\tilde{n}(b)+\tilde{n}(c)$, and equality holds under some conditions.
We here consider the shifting prime condition and state a different version of difference analogue of the Stothers--Mason theorem in terms of ${\rm rad}_\Delta(P)$ defined by~\eqref{5.002} as follows.

\begin{theorem}\label{abc.theorem}
Let $a$, $b$ and $c$ be relatively shifting prime polynomials such that
    $$
    a+b=c
    $$
and such that $a$, $b$ and $c$ are not all constant. Then
    \begin{equation}\label{diff-abc.eq}
    \max\{\deg(a),\deg(b),\deg(c)\}\leq \deg({{\rm rad}_\Delta(abc)})-1=\tilde{n}(abc)-1.
    \end{equation}
\end{theorem}

In order to prove Theorem \ref{abc.theorem}, we need the following lemma corresponding to  the result~\cite[Lemma~2.1]{IKLT}, which reveals the degree relation between $P$ and $\text{gcd}(P,\Delta P)$ when $P$ is a polynomial.
\begin{lemma}\label{deg.lemma}
Let $P$ be a nonzero polynomial. Then
\begin{equation}
\deg(P)=\deg({\rm gcd}(P,\Delta P))+\deg({\rm rad}_{\Delta}(P)).\label{5.004}
\end{equation}
\end{lemma}
\begin{proof}
By \eqref{5.001}, $P$ can be written by $P(z)=A\prod_{j=1}^N(z-z_j)^{\underline{n_j}}$ with a constant $A\ne0$.
It follows from the arguments in Corollary~\ref{n-1.cor},
$\Delta P$ does not have factor $z-z_j+n_j-1$, $j=1, 2, \dots, N$. Hence we have
\begin{equation}
{\rm gcd}(P(z),\Delta P(z))=\prod_{j=1}^m(z-z_j)^{\underline{n_j-1}}.\label{5.005}
\end{equation}
By \eqref{5.002} and \eqref{5.005}, we obtain \eqref{5.004}. We have thus proved Lemma~\ref{deg.lemma}.
\end{proof}

We also need to recall a summation property, see e.g., \cite[Page 48]{KP2001}, \cite[Pages 115--116]{Kohno1999}, \cite[327--328]{M-Thomson1933}.
\begin{lemma}\label{summation}
Let $R(z)$ be a rational function.
We write $R(z)$ in the form
\begin{eqnarray*}
R(z)=\rho\ \frac{\prod_{k=1}^n (z-\alpha_k)}{\prod_{j=1}^m (z-\beta_j)},
\end{eqnarray*}
where $\rho\ne0$, $\alpha_k$, $k=1, \dots, n$ and $j=1, \dots, m$ are complex numbers.
The first order linear homogeneous equation
\begin{eqnarray*}
y(z+1)=R(z)y(z)
\end{eqnarray*}
can be solved as
\begin{eqnarray}
y(z)=\pi(z)\rho^z\ \frac{\prod_{k=1}^n \Gamma(z-\alpha_k)}{\prod_{j=1}^m \Gamma(z-\beta_j)},\label{5.006}
\end{eqnarray}
where $\pi(z)$ is an arbitrary periodic function of period $1$.
\end{lemma}

\bigskip
\noindent\textit{Proof of Theorem \ref{abc.theorem}}.
We may suppose that $\deg(c)=\max\{\deg(a),\deg(b),\deg(c)\}$ without loss of generality.
Since $a+b=c$, we have $\Delta a+\Delta b=\Delta c$. Multiplying the first equation by $\Delta a$, the second by $a$, and then subtracting them, we obtain
    \begin{equation}\label{abc.eq}
    b\Delta a-a\Delta b=c\Delta a-a\Delta c.
    \end{equation}
We see that $\text{gcd}(a,\Delta a)$, $\text{gcd}(b,\Delta b)$, $\text{gcd}(c,\Delta c)$ all divide $b\Delta a-a\Delta b$ by \eqref{abc.eq}.
Since $a$, $b$ and $c$ are relatively shifting prime, from Lemma
\ref{shifting prime} $\text{gcd}(a,\Delta a)$, $\text{gcd}(b,\Delta b)$, $\text{gcd}(c,\Delta c)$ are relatively prime. It yields that they are the factors of $b\Delta a-a\Delta b$. Then we have
    $$
    \deg(\text{gcd}(a,\Delta a))+\deg(\text{gcd}(b,\Delta b))
    +\deg(\text{gcd}(c,\Delta c))\leq \deg(a)+\deg(b)-1
    $$
provided that $b\Delta a-a\Delta b\ne0$.
We add $\deg(c)$ to both sides, and we apply Lemma~\ref{deg.lemma},
\begin{align*}
    \deg(c)&\leq \deg(a)-\deg(\text{gcd}(a,\Delta a))+
    \deg(b)-\deg(\text{gcd}(b,\Delta b))\\
    &\ \ \ +\deg(c)-\deg(\text{gcd}(c,\Delta c))-1\\
    &\leq \deg({{\rm rad}_\Delta(a)})+\deg({{\rm rad}_\Delta(b)})+\deg({{\rm rad}_\Delta(c)})-1.
\end{align*}
By means of Remark~\ref{shifting radical}, we obtain \eqref{diff-abc.eq} when $b\Delta a-a\Delta b\ne0$.\vspace{0.25cm}

Below we consider the case when $b\Delta a-a\Delta b=0$. From \eqref{abc.eq},
    $$
    \frac{\Delta a}{a}=\frac{\Delta b}{b}=\frac{\Delta c}{c}=\kappa(z),
    $$
where $\kappa(z)$ is a rational function. Then polynomials $a$, $b$ and $c$ are solutions of the first order linear difference equation $\Delta f(z)=\kappa(z) f(z)$, namely
    \begin{equation}\label{orderone.eq}
    f(z+1)=(\kappa(z)+1)f(z).
    \end{equation}
By means of Lemma~\ref{summation}, there exists a nonconstant polynomial $p$ so that $a=\pi_1 p$, $b=\pi_2 p$ and $c=\pi_3 p$, where $\pi_j$, $j=1, 2, 3$ are periodic functions with period $1$. Since a periodic function is a transcendental function or a constant, the only possibility is that $\pi_j$, $j=1, 2, 3$ are all constants.
We claim that every shifting zero of $p$ is simple height. Otherwise, $a$, $b$ and $c$ are not relatively shifting prime. Hence, we have
    $$
    \deg({\rm rad}_\Delta(abc))=\deg(a)+\deg(b)+\deg(c).
    $$
It is easy to see that inequality in \eqref{diff-abc.eq} holds in this case. The right equality follows from Proposition~\ref{radical degree}. We have thus proved Theorem~\ref{abc.theorem}.\qquad $\square$

\begin{remark}\label{sharpm2}
We could see the assertion of Theorem \ref{abc.theorem} is sharp. For example, let $a(z)=z(z-1)$, $b(z)=-(z-4)(z-5)$ and $c(z)=4(2z-5)$. We see that $a$, $b$ and $c$ are relatively shifting prime polynomials satisfying $a+b=c$. In addition, it is shown that $\max\{\deg(a),\deg(b),\deg(c)\}=2$, ${\rm rad}_\Delta(abc)=z(z-5/2)(z-4)$, and $\deg({\rm rad}_\Delta(abc))=3$.
\end{remark}
\vspace{0.25cm}
\section{Extension of the Stothers--Mason theorem with difference radical}\label{edsm}

We write $P$ in the form \eqref{5.001}, i.e., $P(z)=A\prod_{j=1}^N(z-z_j)^{\underline{n_j}}$.
The difference radical of truncation level $q\in\N$ for a polynomial $P(z)$ which is denoted by
    \begin{equation}\label{radq.eq}
    {\rm rad}_\Delta^q(P)={\rm gcd}\bigg(\prod_{j=1}^N(z-z_j)^{\underline{n_j}}, \ \prod_{j=1}^N(z-z_j)^{\underline{q}}\bigg).
    \end{equation}
It is a generalization of difference radical. When $q=1$ in \eqref{radq.eq}, it reduces to the difference radical of $P$ given in \eqref{5.002}.
We note that ${\rm rad}_\Delta^q(P)={\rm rad}_\Delta (P)$ holds for any $q\geq 2$ if and only if all the shifting zeros of $P$ are simple height.
It follows from \eqref{radq.eq},
\begin{equation}
\deg({\rm rad}_\Delta^q(P))=\deg\left(\prod_{j=1}^N(z-z_j)^{\underline{\min(n_j,q)}}\right)=\sum_{j=1}^N \min(n_j,q).\label{6.0017}
\end{equation}
In addition, we have an inequality
    \begin{equation}\label{small.eq}
   \deg\left( {\rm rad}_\Delta^q(P)\right)\leq q\cdot\deg \left({\rm rad}_\Delta(P)\right).
    \end{equation}


The following theorem extends Theorem~\ref{abc.theorem} for $m+1$ polynomials, where $m\in\N$, $m\geq2$. We state it as follows.
\begin{theorem}\label{fm.theorem}
Let $m\in \N$, $m\geq2$ and let $f_1,\ldots,f_{m+1}$  be pairwise relatively shifting prime polynomials with $\min_{1\leq i\leq m+1} \deg f_i\geq m-1$ satisfying the following functional equation
    \begin{equation}\label{a1am}
    f_1+ \cdots + f_m=f_{m+1},
    \end{equation}
and such that $f_1,\dots,f_{m}$ are linearly independent over $\C$. Then
    \begin{equation}\label{diffMasonineq}
    \begin{split}
    \max_{1\leq i\leq m+1}\{\deg f_i\}& \leq \deg({\rm rad}_{\Delta}^{m-1}(f_1f_2\cdots f_{m+1}))-\frac{1}{2}m(m-1)\\
    &\leq (m-1)\deg({\rm rad}_{\Delta}(f_1f_2\cdots f_{m+1}))-\frac12m(m-1).
    \end{split}
    \end{equation}
\end{theorem}
%

In order to prove Theorem~\ref{fm.theorem}, we need two lemmas below.
We first consider a generalization of Lemma~\ref{deg.lemma}. For $x\in\R$, we write $[x]^+=\max\{x,0\}$.

\begin{lemma}\label{deg2.lemma}
Let $n\in \N$ and let $P$ be a polynomial with $\deg P\geq n$ represented by \eqref{5.001}. Then
\begin{equation}
\deg({\rm gcd}(P,\Delta P, \dots, \Delta^{n} P))=\sum_{j=1}^N [n_j-n]^+\label{6.0015}
\end{equation}
and
\begin{equation}\label{6.0014}
\deg(P)-\deg({\rm gcd}(P,\Delta P,\ldots,\Delta^{n} P))=\deg({\rm rad}_{\Delta}^{n}(P)).
\end{equation}
\end{lemma}

\begin{proof}
Let $z_j$ be a shifting zero of $P$ with height $n_j$ such that $P(z_j-1)\ne0$. We write $P(z)=(z-z_j)^{\underline{n_j}}G_j(z)$, where $G_j$ is a polynomial.
From the arguments in the proof of Corollary~\ref{n-1.cor}, $\Delta P(z)=(z-z_j)^{\underline{n_j-1}}Q_{j,1}(z)$, where $Q_{j,1}$ is a polynomial satisfying $Q_{j,1}(z_j+n_j-1)\ne0$. Let $k\in N$, $1\leq k\leq n$.
If $n_j\geq n+1$, then $\Delta^k P(z)=(z-z_j)^{\underline{n_j-k}}Q_{j,k}(z)$, where $Q_{j,k}$ is a polynomial satisfying $Q_{j,k}(z_j+n_j-k)=0$ for any $1\leq k\leq n-1$.
This implies that ${\rm gcd}(P,\Delta P, \dots, \Delta^{n} P)$ is divided by $(z-z_j)^{\underline{n_j-n}}$. If  $n_j\leq n$, then the factors in $(z-z_j)^{\underline{n_j}}$ never contribute to ${\rm gcd}(P,\Delta P, \dots, \Delta^{n} P)$.
Then
\begin{equation}
{\rm gcd}(P,\Delta P, \dots, \Delta^{n} P)=\prod_{j=1}^N (z-z_j)^{\underline{[n_j-n]^+}},\label{6.0035}
\end{equation}
which implies \eqref{6.0015}.
From \eqref{6.0015}, \eqref{5.001} and \eqref{6.0017}, we obtain \eqref{6.0014}. In fact,
\begin{align*}
\deg&({\rm gcd}(P,\Delta P, \dots, \Delta^{n} P))=\sum_{j=1}^N [n_j-n]^+=\sum_{j=1}^N \max(n_j-n,\ 0)\\
&=\sum_{j=1}^N n_j-\sum_{j=1}^N \min(n_j, n)=\deg(P)-\deg({\rm rad}_\Delta^n(P)).
\end{align*}
We have thus proved Lemma~\ref{deg2.lemma}.
\end{proof}

\begin{remark}\label{shifting radical2}
Let $n\in\N$ and let $P$ and $Q$ be nonconstant polynomials.
We have the following estimate as well as Remark~\ref{shifting radical}.
\begin{equation}
\deg({\rm rad}_\Delta^n (PQ))\leq
\deg({\rm rad}_\Delta^n (P))+({\rm rad}_\Delta^n (Q)).\label{4.001}
 \end{equation}
By the definition of the shifting prime, the equality in \eqref{5.0025} holds if $P$ and $Q$ are shifting prime.
\end{remark}


Let $m\in \N$ and let $f_i(z)$, $1\leq i\leq m$ be polynomials.
We next recall the definition of the Casoratian (Casorati determinant), which has played important roles in difference calculus, see e.g., \cite[Pages 354--357]{M-Thomson1933},~\cite[Pages 276--281]{Norlund1924}.
\begin{eqnarray}
{\mathcal C}(z)
={\mathcal C}(f_1,f_2,\dots,f_m)=
\left|
\begin{array}{cccc}
f_1(z) & f_2(z) & \cdots & f_m(z) \\
\Delta f_1(z) & \Delta f_2(z) & \cdots & \Delta f_m(z) \\[2ex]
  & \hdotsfor{2} & \\[2ex]
\Delta^{m-1}f_1(z) & \Delta^{m-1}f_2(z) & \cdots & \Delta^{m-1}f_m(z)
\end{array}
\right|.\label{6.001}
\end{eqnarray}
We need a number of known results in linear algebra e.g.,~\cite{Lang}.
By the properties of the determinant, we have
\begin{lemma}\label{Casoratian} The following claims hold.
\quad
\begin{enumerate}
\item[{\rm (i)}]\enspace It is possible to replace $\Delta^k f_i$ with $f_i(z+k)$, $i=1,2,\dots, m$, $k=1,2,\dots, m-1$ in ${\mathcal C}(z)$.
\item[{\rm (ii)}]\enspace ${\mathcal C}(z)$ does not vanish identically in case $f_i(z)$, $1\leq i\leq m$ are linearly independent over $\C$.
\item[{\rm (iii)}]\enspace It is possible to replace $f_i$ with $f_{m+1}=f_1+f_2+\cdots+ f_m$ for arbitrary fixed $i\in\{1,2,\dots, m\}$ in ${\mathcal C}(z)$.
\item[{\rm (iv)}]\enspace ${\mathcal C}(z)$ is divided by ${\rm gcd}(f_i,\Delta f_i, \dots, \Delta^{m-1} f_i))$, $i=1, 2, \dots, m$, some of which may be constants.
\end{enumerate}
\end{lemma}

The equation \eqref{a1am} was considered in~\cite[Theorem~3.5]{IKLT} by the estimates of degree of Casoratian ${\mathcal C}(z)$ replacing $\Delta^k f_i$ with $f_i(z+k)$, $i=1,2,\dots, m$, $k=1,2,\dots, m-1$ in ${\mathcal C}(z)$, see Lemma~\ref{Casoratian}~(i). Following the idea in~\cite{IKLT}, we prove Theorem~\ref{fm.theorem}.
\vspace{0.25cm}

\noindent{\it Proof of Theorem~\ref{fm.theorem}}\enspace
We consider the Casoratian ${\mathcal C}(z)$ of $f_1(z), \dots , f_m(z)$ defined by \eqref{6.001}.
Since $f_1,\dots,f_{m}$ are linearly independent over $\C$, we have ${\mathcal C}(z)\not\equiv0$ by Lemma~\ref{Casoratian}~(ii).

 By \eqref{a1am} and Lemma~\ref{Casoratian}~(iii), we can replace $f_i$ with $f_{m+1}$ for arbitrary fixed $i\in\{1,2,\dots, m\}$ in the right hand of \eqref{6.001}.

Write $f_i(z)=A_i \prod_{j=1}^{N_i}(z-z_{i,j})^{\underline{n_{i,j}}}$, $f_i(z_{i,j}-1)\ne0$, $1\leq j\leq N_i$, $1\leq i\leq m+1$ following \eqref{5.001}, where $A_i$ are constants, $N_i$, $n_{i,j}\in \N$.
This gives that $f_i(z)$ is divided by ${\rm gcd}(f_i,\Delta f_i, \dots, \Delta^{m-1} f_i))=\prod_{j=1}^{N_i}(z-z_{i,j})^{\underline{[n_{i,j}-m+1]^+}}$ for any $i\in\{1,2,\dots, m+1\}$.
By means of Lemma~\ref{Casoratian}~(iv), we see that the Casoratian ${\mathcal C}(z)$ is divided by ${\rm gcd}(f_i,\Delta f_i, \dots, \Delta^{m-1} f_i))$, $1\leq i\leq m+1$.
%
%
Since we assume that $f_1,\ldots,f_{m+1}$ are pairwise relatively shifting prime, ${\mathcal C}(z)$ is divided by $q(z):=\prod_{i=1}^{m+1}\prod_{j=1}^{N_i}(z-z_{i,j})^{\underline{[n_{i,j}-m+1]^+}}$, namely there exists a polynomial $p(z)$ such that ${\mathcal C}(z)=p(z)q(z)$.

%
By the assumption $\min_{1\leq i\leq m} \deg f_i\geq m-1$, we have $\deg(\Delta^k f_i)=\deg f_i-k\geq 0$, $0\leq k\leq m-1$.
This implies that $\deg {\mathcal C}(z)$ is never beyond any sum of distinct $m$ of the $\deg f_i(z)$, $1\leq i\leq m+1$ minus $\sum_{k=0}^{m-1}k=m(m-1)/2$ as the sum of $\deg\bigl(\Delta^{k}f_{j_{\nu}}\bigr)$ for the mutually distinct $m$ integers $j_{\nu}\in\{1, \dots , m, m+1\}$.
Since $m\in \N$ and $f_1,\ldots,f_{m+1}$ are pairwise relatively shifting prime, by Lemma~\ref{deg2.lemma} and Remark~\ref{shifting radical2} we have
\begin{align*}
\min_{1\leq h\leq m+1}& \sum_{1\leq i\leq m+1, i\neq h} \deg f_i -\frac{1}{2}m(m-1)
\geq \deg\left(\prod_{i=1}^{m+1}{\rm gcd}(f_i,\Delta f_i, \dots, \Delta^{m-1} f_i))\right)\\
&=\sum_{i=1}^{m+1}\left(\deg (f_i)-\deg({\rm rad}_{\Delta}^{m-1}(f_j))\right)=\sum_{i=1}^{m+1}\deg (f_i)-\deg({\rm rad}_{\Delta}^{m-1}(f_1f_2\cdots f_{m+1})),
\end{align*}
which implies the first inequality of \eqref{diffMasonineq}.
Combining \eqref{small.eq} and the first inequality of \eqref{diffMasonineq}, we see that the second inequality of \eqref{diffMasonineq} follows. We have proved Theorem~\ref{fm.theorem}. \qquad $\square$

\begin{example}\label{exfor4.1}\enspace
Let $c\in \C\setminus\{0\}$, $k\in\C\setminus\{0,\pm5\}$. Set $\alpha=c+k/10-1/2$ and $\beta=c-k/10-1/2$.
We consider polynomials
$f_1(z)=(z-\alpha)^{\underline{5}}$, $f_2(z)=-(z-\beta)^{\underline{5}}$,
$f_3(z)=k(z-c)^{\underline{4}}$,
which are are linearly independent over $\C$, and consider a polynomial
$$
f_4(z)=-\frac{k(k-5)(k+5)}{50000}\left(1000z^2-(2000c+3000)z+k^2+1000c^2+3000c+1775\right).
$$
We see that $f_1$, $f_2$, $f_3$ and $f_4$ are pairwise relatively shifting prime for some suitable $c$ and $k$, and they satisfy \eqref{a1am} of the case $m=3$, i.e., $f_1+f_2+f_3=f_4$.
We have by definition, $\max_{1\leq i\leq 4}\{\deg f_i\}=5$, and $\deg{\rm rad}_\Delta^2 f_1=2$,
$\deg{\rm rad}_\Delta^2 f_2=2$, $\deg{\rm rad}_\Delta^2 f_3=2$, $\deg{\rm rad}_\Delta^2 f_4=2$ by \eqref{6.0017}. This gives that the first inequality of \eqref{diffMasonineq} is sharp in Theorem~\ref{fm.theorem}.
\end{example}

\begin{remark}\enspace Example~\ref{exfor4.1} shows that the inequality in Theorem~3.5 of {\rm\cite{IKLT}} is sharp for the case when $m=3$.
\end{remark}

\section{Polynomial solutions of difference Fermat functional equations}\label{psf}

In this section, we apply Theorem~\ref{abc.theorem} and Theorem~\ref{fm.theorem} to difference Fermat type functional equations for investigating nonexistence of polynomial solutions.
We adopt the falling expression $P^{\underline{n}}$ instead of $P^n$ for a polynomial $P$ in the Fermat type functional equations, which is the same to the case we adopt rising expressions $P^{\overline{n}}$ in some sense. For the Fermat type functional equations, see e.g., \cite{GH2004},~\cite{NS1979}.
Concerning the methods of the proofs, we follow the idea in~\cite{IKLT}.

\begin{theorem}\label{Equ3}
Let $a$, $b$ and $c$ be polynomials, not all constants, and $n\in\N$ such that $a^{\underline{n}}$, $b^{\underline{n}}$ and $c^{\underline{n}}$ are relatively shifting prime and satisfy
    \begin{equation}\label{abc.equ}
    a^{\underline{n}}+b^{\underline{n}}=c^{\underline{n}}.
    \end{equation}
Then $n\leq 2$. In addition, if one of $a$, $b$ and $c$ is a constant, then $n=1$.
\end{theorem}

\begin{proof}
Let us first assume that $a$, $b$ and $c$ are all nonconstant polynomials.
By Theorem~\ref{abc.theorem} and Remark~\ref{shifting radical}, we have
    \begin{equation}\label{a.eq}
    \begin{split}
    n\deg(a)=\deg(a^{\underline{n}})&\leq\max\{ \deg(a^{\underline{n}}), \deg(b^{\underline{n}}),\deg(c^{\underline{n}})\}\leq \deg({\rm rad}_\Delta(a^{\underline{n}}~b^{\underline{n}}~c^{\underline{n}}))-1\\
    &\leq\deg({\rm rad}_\Delta(a^{\underline{n}}))+
    \deg({\rm rad}_\Delta(b^{\underline{n}}))+
    \deg({\rm rad}_\Delta(c^{\underline{n}}))-1\\
    &\leq \deg(a)+\deg(b)+\deg(c)-1.
    \end{split}
    \end{equation}
Similarly, for $b$ and $c$, we have
    \begin{equation}\label{b.eq}
    n\deg(b)\leq \deg(a)+\deg(b)+\deg(c)-1
    \end{equation}
and
    \begin{equation}\label{c.eq}
    n\deg(c)\leq \deg(a)+\deg(b)+\deg(c)-1.
    \end{equation}
Combining inequalities \eqref{a.eq}, \eqref{b.eq} and \eqref{c.eq}, we obtain
    $$
    n(\deg(a)+\deg(b)+\deg(c))\leq 3(\deg(a)+\deg(b)+\deg(c))-3,
    $$
which implies that $n\leq 2$.

We secondly assume one of $a$, $b$ and $c$ is a constant. Without loss of generality, we assume that $c$ is a constant. Then \eqref{a.eq} and \eqref{b.eq} yield
    $
    n(\deg(a)+\deg(b))\leq 2(\deg(a)+\deg(b))-2,
    $
which shows that $n\leq 1$. We proved our assertion.
\end{proof}

We have an example for the case $n=2$, see~\cite[Example~4.2]{IKLT}.
\begin{example}\label{deg2example}
We consider polynomials $a(z)=z^2$, $b(z)=-(i/2)\left(\sqrt{2}z^2 + 2z -\sqrt{2}\right)$ and $c(z)=-(1/2)\left(\sqrt{2}z^2 - 2z -\sqrt{2}\right)$. These polynomials satisfy \eqref{Equ3}.
The zeros of $b$ and $c$ are $\{-(1+\sqrt{3})/\sqrt{2},\ (-1+\sqrt{3})/\sqrt{2}\}$ and $\{(1-\sqrt{3})/\sqrt{2},\ (1+\sqrt{3})/\sqrt{2}\}$, respectively. Hence they are relatively shifting prime.
\end{example}

The next result extends Theorem \ref{Equ3} to equations
with arbitrarily many terms.

\begin{theorem}\label{Equn.eq}
Let $m\in \N$, $m\geq 2$, $n\in\N$. Suppose that there exist $f_1,\ldots,f_{m+1}$ nonconstant polynomials satisfying
\begin{equation}
    f_1^{\underline{n}}+f_2^{\underline{n}}+\cdots
    +f_m^{\underline{n}}=f_{m+1}^{\underline{n}}.\label{5.51}
\end{equation}
Further suppose  that $f_1^{\underline{n}}$, $\ldots$, $f_{m+1}^{\underline{n}}$ are pairwise relatively shifting prime and $f_1^{\underline{n}},\ldots,f_{m}^{\underline{n}}$ are linearly independent.
Then
\begin{equation}
n\leq m^2-1-\frac{m(m-1)}{2\max_{1\leq i\leq m+1}\{\deg(f_i)\}}.
\end{equation}
\end{theorem}
\begin{proof}
By using Theorem \ref{fm.theorem}, we have
\begin{align*}
 n\max_{1\leq i\leq m+1}\{\deg(f_i)\}
& \leq(m-1)\deg({\rm rad}_\Delta(f_1^{\underline{n}}f_2^{\underline{n}}\cdots f_{m+1}^{\underline{n}})-\frac{1}{2}m(m-1)\\
& \leq(m+1)(m-1)\max_{1\leq i\leq m+1}\{\deg(f_i)\}-\frac{1}{2}m(m-1).
\end{align*}
%
Since $\max_{1\leq i\leq m+1}\{\deg(f_i)\}\geq 1$, we obtain our assertion.
\end{proof}

\begin{remark}\enspace We change the assumption that $f_{m+1}(z)$ is a nonconstant polynomial in Theorem~\ref{Equn.eq} into $f_{m+1}(z)\equiv1$, and keep other conditions. We use the similar arguments in the proof of Theorem~\ref{Equn.eq} with $\deg (f_{m+1})=0$. Then we obtain an estimate for polynomial solutions of $ f_1^{\underline{n}}+f_2^{\underline{n}}+\cdots +f_m^{\underline{n}}=1$,
\begin{equation*}
 n\max_{1\leq i\leq m}\{\deg(f_i)\}\leq
    m(m-1)\max_{1\leq i\leq m}\{\deg(f_i)\}-\frac{1}{2}m(m-1),
\end{equation*}
and hence $n\leq m(m-1)\left(1-1/(2\max_{1\leq i\leq m}\{\deg(f_i)\})\right)$. This implies that $n\leq m^2-m-1$, which corresponds to the result on polynomial solutions of $f_1^{n}+f_2^{n}+\cdots+f_m^{n}=1$, see~e.g.,~{\rm\cite{GH2004},~\cite{PEMS2020}}. When $m=3$, we have $n\leq5$. There exist examples for the case $n=2$ and $n=3$ below.
\end{remark}

\begin{example}\enspace Consider relatively shifting prime polynomials $f_1(z)=(1/\sqrt{2})z+1$,
$f_2(z)=(1/2)z+(1/2)(\sqrt{2}-\sqrt{6})$ and $f_3(z)=(\sqrt{3}i/2)z+(i/2)(\sqrt{6}-\sqrt{2})$. By computations, we see that $f_1$, $f_2$ and $f_3$ satisfy
\begin{equation}
f_1^{\underline{2}}+f_2^{\underline{2}}+f_3^{\underline{2}}=1.\label{5.52}
\end{equation}
We consider another triad of relatively shifting prime polynomials $f_1(z)=(1/24\sqrt{2})(24z^2+48z-29)$, $f_2(z)=(1/48)(24z^2-48z-61)$ and $f_3(z)=(i/16\sqrt{3})(24z^2+16z+3)$.
These polynomials satisfy \eqref{5.52}.
\end{example}

\begin{example}\enspace Let $s\in\C$ be a root of $s^9-144s^3+108=0$ and $t\in\C\setminus\{0\}$.
We consider polynomials
$f_1(z)=z^3-a_2 z^2-a_1 z+a_0$, $f_2(z)=-z^3+a_2 z^2+a_1 z-(3 a_0+s^3)/3$
 and $f_3(z)=sz^2+tz+(t^2-4s^2)/4s$,
where
$$
a_2=-\frac{3t}{2s},\ a_1=\frac{3(4s^2-t^2)}{4s^2},\ a_0=\frac{3t^3-36s^2t-4s^6}{24s^3}
$$
By computations, we see that $f_1$, $f_2$ and $f_3$ satisfy
\begin{equation*}
f_1^{\underline{3}}+f_2^{\underline{3}}+f_3^{\underline{3}}=1,\label{5.53}
\end{equation*}
for an arbitrary $t$. We thus obtain $f_1$, $f_2$ and $f_3$ that are relatively shifting prime polynomials for a suitable $t$.
\end{example}

\bigskip

\medskip
\noindent
\emph{Katsuya Ishizaki}\\
\textsc{The Open University of Japan, 2-11 Wakaba,\\
Mihama-ku, Chiba, 261- 8586 Japan}\\
\texttt{email:ishizaki@ouj.ac.jp}

\medskip
\noindent
\emph{Z.-T.~Wen}\\
\textsc{Shantou University, Department of Mathematics,\\
Daxue Road No.~243, Shantou 515063, China}\\
\texttt{e-mail:zhtwen@stu.edu.cn}

\vspace{1cm}
\end{document}